\documentclass{article}

\usepackage{arxiv}
\usepackage[utf8]{inputenc} 
\usepackage[T1]{fontenc}    
\usepackage{hyperref}       
\usepackage{url}            
\usepackage{booktabs}       
\usepackage{amsfonts}       
\usepackage{nicefrac}      
\usepackage{microtype}      
\usepackage{graphicx}
\usepackage{natbib}
\usepackage{doi}
\usepackage{subcaption}
\usepackage{algorithm}
\usepackage{algorithmicx}
\usepackage{algpseudocode}
\usepackage{amsmath,amsfonts,amssymb}
\usepackage{multirow}
\usepackage{amsthm}
\usepackage{xcolor}

\newtheorem{proposition}{Proposition}

\title{Efficient and Fast Tensor-Product Multinode Shepard Collocation for Elliptic PDEs on Cartesian Grids}

\date{}

\author{ \href{https://orcid.org/0000-0001-7608-5635}{\includegraphics[scale=0.06]{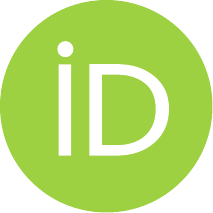}\hspace{1mm}Anouar El~Harrak}\thanks{Corresponding author.} \\
	SMMADCSN Laboratory, FPL \\
    Abdelmalek Essaadi University\\
	Tetouan, Morocco \\
	\texttt{a.elharrak@uae.ac.ma} \\
	\And
	\href{https://orcid.org/0000-0002-5988-650X}{\includegraphics[scale=0.06]{orcid.pdf}\hspace{1mm}Hatim Tayeq} \\
	SMMADCSN Laboratory, FPL \\
    Abdelmalek Essaadi University\\
	Tetouan, Morocco \\
	\texttt{h.tayeq@uae.ac.ma} \\
	\And
	\href{https://orcid.org/0000-0003-2314-6873}{\includegraphics[scale=0.06]{orcid.pdf}\hspace{1mm}Benaissa Zerroudi} \\
	Laboratory of Engineering Sciences \\
    Ibn Zohr University, 80080\\
    Agadir, Morocco\\
	\texttt{zerroudi@gmail.com} \\
}

\renewcommand{\shorttitle}{}
\begin{document}
\maketitle

\begin{abstract}
We introduce a Grid-Based Multinode Shepard Collocation Method (GBMSC) for two-dimensional elliptic boundary value problems on rectangular domains. The method combines Shepard-type partition functions with tensor-product Lagrange interpolation on overlapping local Cartesian subgrids. The Cartesian structure avoids the local unisolvency search required by multinode Shepard constructions on scattered data, while the local character of the approximation yields sparse collocation matrices and moderate conditioning. We establish the main structural properties of the method, including the cardinal property, tensor-product polynomial reproduction, and local interpolation estimates. Numerical experiments for Poisson problems with Dirichlet and mixed boundary conditions confirm the polynomial reproduction property up to round-off accuracy and show regular convergence for smooth non-polynomial solutions. Comparisons with Multinode Shepard Collocation Method (MSC) and Kansa's collocation methods indicate that the proposed discretization preserves high accuracy with significantly better conditioning than the RBF-based schemes considered.
\end{abstract}


\section{Introduction}
Partial differential equations (PDEs) are a standard mathematical tool for the description of phenomena arising in science and engineering, including climate modelling \cite{ghil2020physics,kalnay2003atmospheric}, air quality forecasting \cite{seinfeld2016atmospheric,tayeq2021self}, medical imaging \cite{aubert2006mathematical,chan2005image}, financial market modelling \cite{achdou2012partial,el2023pdes}, and aerospace engineering \cite{petrila2005basics,mavris2012overview}. Since explicit solutions are rarely available for models of interest, the construction of accurate and stable numerical methods remains a central topic in numerical analysis.

Classical discretization techniques such as the finite element method and beyond are supported by a well-developed mathematical theory and are widely used for the numerical solution of PDEs \cite{alves2021numerical,ciarlet2002finite,quarteroni2008numerical,leveque2007finite,thomee2006galerkin}. However, their implementation may become demanding when dimension increases, or when the generation and management of mesh connectivity become computationally expensive \cite{alves2021numerical, tayeq2026optimizedd}. Meshless methods were introduced, in part, to reduce this dependence on structured connectivity \cite{fasshauer2007meshfree,buhmann2003radial,wendland2005scattered}. Kansa's method approximates the unknown solution by radial basis functions and imposes the differential equation together with the boundary conditions at the collocation nodes \cite{kansa1990multiquadricsI,kansa1990multiquadricsII}. This meshless formulation is useful for irregular point distributions, but it usually leads to dense collocation matrices \cite{fasshauer2007meshfree,fornberg2015primer}. In addition, the conditioning may become highly ill-conditioned and the accuracy is sensitive to the shape of the radial basis function \cite{fornberg2015primer,schaback1995error,fornberg2004flyer}.

Shepard-type methods provide another family of meshless approximation techniques. The classical Shepard operator interpolates scattered data using normalized inverse-distance weighting \cite{shepard1968two,zerroudi2024scattered}. Although it has a simple structure and good qualitative properties, its polynomial reproduction is limited to constants, and its approximation order is therefore low. Several modifications have been proposed to improve the approximation order by replacing function values with local polynomial interpolants and combining them through Shepard-type basis functions \cite{cavoretto2025shape,dell2024multinode,kocc2020finite,zerroudi2024approximate,zhang2021finite}. This idea leads to multinode Shepard-type constructions, which preserve the partition of unity structure and increase the polynomial precision.

More recently, Multinode Shepard Collocation method has been applied to elliptic PDEs and has shown advantages with respect to RBF-based discretizations, especially in terms of conditioning of the collocation matrix and accuracy of the computed solution \cite{dell2024multinode,dell2024enriched,dell2021solving}. However, in a scattered-data setting, the construction of local polynomial interpolants requires the selection of unisolvent subsets. This step may be computationally expensive and may fail when the local geometry is not suitable for the prescribed polynomial space.

The present paper proposes a grid-based Shepard construction designed for data arranged on Cartesian grids. The method combines Shepard-type weights with tensor-product Lagrange interpolation \cite{occorsio2022lagrange, shaalini2021new} on overlapping local Cartesian subgrids. In this way, the local polynomial interpolants are obtained directly from the grid structure, without solving local unisolvency problems on scattered nodes. The method is particularly natural on rectangular domains, where the boundary is explicitly described and the local tensor-product structure can be used without additional geometric processing.

The proposed GBMSC method is used for two-dimensional elliptic boundary value problems. The global approximation is written in terms of basis functions obtained by combining local Lagrange polynomials expressed in tensor-product form with Shepard weights. The derivatives of the basis are assembled from local differentiation matrices and derivatives of the weights. This construction leads to sparse discrete operators, since each matrix entry receives contributions only from local subgrids containing the corresponding pair of nodes. Moreover, the structure of the tensor-product gives exact reproduction of polynomials.

The numerical results examine Poisson problems with Dirichlet and mixed boundary conditions. We use polynomial solutions to validate the reproduction property and smooth non-polynomial solutions to study convergence, conditioning, sparsity, and the influence of the overlap parameters. Then, we make a comparison with the multinode Shepard method and Kansa's methods in terms of accuracy and condition numbers.

In Section 2, we introduce the grid-based Shepard construction in rectangular domains, define the construction of global basis functions, and establish the main properties of the method, including cardinality, polynomial reproduction, and local interpolation estimates. Section 3 describes the discretization of elliptic problems and the sparse assembly of the corresponding matrices, presents numerical experiments, and discusses the influence of polynomial degree, overlap, and grid refinement on accuracy and conditioning.

\section{Grid-Based Multinode Shepard Collocation Method on Rectangular Domains} 
\subsection{Construction of Overlapping Cartesian Subgrids}
Let $\mathcal{D}=[a,b]\times[c,d]\subset\mathbb{R}^2$ and let 
\[
\mathcal{X}=\{\mathbf{x}_{i,j}=(x_i,y_j):\ i=1,\dots,N_x,\ \ j=1,\dots,N_y\}, 
\]
be a Cartesian grid on $\mathcal{D}$, composed of $N = N_x \times N_y$ points, such that 
\[
a=x_1<x_2<\ldots<x_{N_x}=b\quad\text{and}\quad c=y_1<y_2<\ldots<y_{N_y}=d.
\]

The construction of our new method is based on a family of local Cartesian subgrids.
In sliding-subgrid construction, a local subgrid of size $n_x\times n_y$, with $n_x$ points in the $x$-direction and $n_y$ points in the $y$-direction, is obtained by fixing a starting index $(i,j)$ and setting
\[
\mathcal{X}_{i,j}
=
\{\mathbf{x}_{i+r,j+s}=(x_{i+r},y_{j+s}):\
r=0,\ldots,n_x-1,\ \ s=0,\ldots,n_y-1\},
\]
such that
\[i+n_x-1\leq N_x \quad \text{and}\quad j+n_y-1\leq N_y.\]
The corresponding geometric support is the rectangle
\[
\Omega_{i,j}
=
[x_i,x_{i+n_x-1}]
\times
[y_j,y_{j+n_y-1}].
\]

Let $(\rho_x,\rho_y)$ be the step used to select consecutive local subgrids. Then, Figure~\ref{fig:sliding_subgrids}, the right and upper neighboring subgrids of $\mathcal X_{i,j}$ are, respectively,
\[
\mathcal X_{i+\rho_x,j}
\qquad\hbox{and}\qquad
\mathcal X_{i,j+\rho_y}.
\]
Hence, the numbers of common grid points in the two coordinate directions are
\[
\eta_x
=
|\mathcal X_{i,j}\cap\mathcal X_{i+\rho_x,j}|_x
=
n_x-\rho_x\quad\text{and}\quad\eta_y
=
|\mathcal X_{i,j}\cap\mathcal X_{i,j+\rho_y}|_y
=
n_y-\rho_y.
\]
Here, $|\cdot|_x$ and $|\cdot|_y$ denote the numbers of grid points in the horizontal and vertical directions, respectively. If the same overlap is imposed in both directions, we write
\[
\eta=\eta_x=\eta_y.
\]

\begin{figure}
    \centering
    \includegraphics[width=0.5\linewidth]{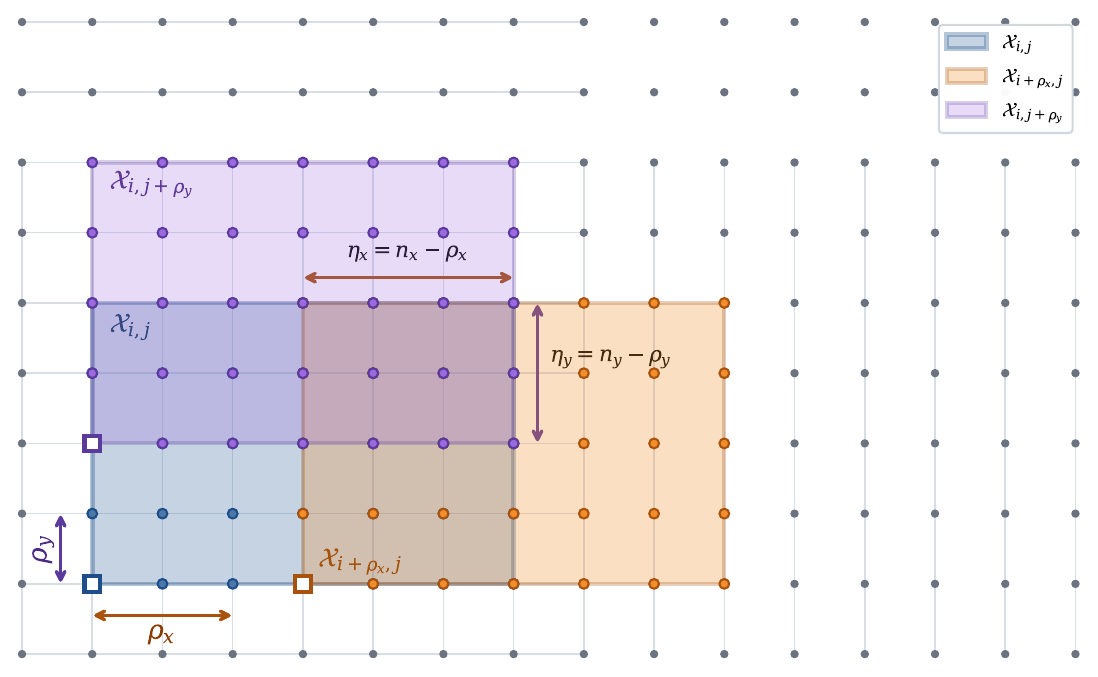}
    \caption{Overlapping between neighboring local subgrids. The blue rectangle represents the current subgrid, while the orange and purple rectangles represent the right and upper neighboring subgrids $\mathcal X_{i+\rho_x,j}$ and $\mathcal X_{i,j+\rho_y}$, respectively.}
    \label{fig:sliding_subgrids}
\end{figure}

The case $\eta=0$ corresponds to consecutive local subgrids without common grid points, whereas $0<\eta<n_x$ means that neighboring subgrids share $\eta$ points in each coordinate direction. In the fully sliding case,
\[
(\rho_x,\rho_y)=(1,1),
\]
we obtain
\[
\eta_x=n_x-1
\quad \text{and}\quad 
\eta_y=n_y-1.
\]
Therefore, the fully sliding construction gives the maximal overlap between consecutive sampled subgrids.

Let
\[
\{i_1,\ldots,i_{M_x}\},
\quad \text{and}\quad 
\{j_1,\ldots,j_{M_y}\}
\]
be the admissible starting indices of the sliding subgrids in the $x$- and $y$-directions, respectively. Thus, each local subgrid is first identified by a pair $(i_{\mathbf  p},j_\mathbf{q})$ and is denoted by
\[
\mathcal{X}_{i_{\mathbf p},j_\mathbf q}.
\]
For notational convenience, we replace this double index by a single index
\[
k=({\mathbf p}-1)M_y+\mathbf q,
\quad {\mathbf p}=1,\ldots,M_x,\quad \mathbf q=1,\ldots,M_y,
\]
and we set
\[
\mathcal{X}_k:=\mathcal{X}_{i_{\mathbf p},j_\mathbf q}.
\]
Consequently, the family of sampled local subgrids can be written equivalently as
\[
\{\mathcal{X}_{i_{\mathbf p},j_\mathbf q}:{{\mathbf p}=1,\ldots,M_x;\,\mathbf q=1,\ldots,M_y}\}
=
\{\mathcal{X}_k\}_{k=1}^{M},
\qquad M=M_xM_y.
\]
By construction, we assume that they satisfy
\[
\bigcup_{k=1}^M \mathcal{X}_k=\mathcal{X},
\]
and that adjacent subgrids overlap. The overlap is introduced in order to guarantee the interaction among neighboring local interpolants and to make possible the construction of a global approximant with the desired smoothness and stability properties.

\subsection{Construction of Local Interpolant}
We denote by
\[
\{u_{i,j}:\ i=1,\dots,N_x,\ \ j=1,\dots,N_y\}
\]
the values of a function $u:\mathcal{D}\rightarrow \mathbb{R}$ at the grid nodes $\mathbf{x}_{i,j}\in\mathcal{X}$.
Depending on the application, these values may represent exact nodal data or sampled measurements.

For every $k=1,\dots,M$, we define the local interpolation operator
\begin{equation}
\mathcal{I}_k[u](\mathbf{x})
=
\sum_{r=1}^{n_x}\sum_{s=1}^{n_y} u_{r,s}^k\, \Phi_{r,s}^k(\mathbf{x}),
\qquad \mathbf{x}=(x,y)\in\mathbb{R}^2,
\label{eq:Loca_Int_rew}
\end{equation}
where 
\[
\Phi_{r,s}^k(\mathbf{x})=\ell_{x,r}^k(x)\,\ell_{y,s}^k(y),
\]
$\ell_{x,r}^k$, $\ell_{y,s}^k$ are the one-dimensional Lagrange basis polynomials associated with the nodes
\[
\{x_r^k\}_{r=1}^{n_x},\qquad \{y_s^k\}_{s=1}^{n_y}, 
\]
and $u_{r,s}^k$ denotes the value of $u$ at the corresponding local nodes $(x_{r}^k,y_s^k) \in \mathcal{X}_k$. Thus, $\mathcal{I}_k[u]$ is the tensor-product Lagrange interpolant on the local grid $\mathcal{X}_k$.

Let $C^{n_x,\,n_y}(\Omega_k)$ be the space of functions
$u : \Omega_k \to \mathbb{R}$ whose partial derivatives
$\partial_x^\alpha \partial_y^\beta u$ exist and are continuous on
$\Omega_k$ for all integers $\alpha, \beta$ satisfying
$0 \le \alpha \le n_x$ and $0 \le \beta \le n_y$.
\begin{proposition}
Let $\Omega_k$ be the geometric support of the subgrid $\mathcal X_k$ of size $n_x\times n_y$ with $n_x=m_x+1$ and $n_y=m_y+1$, and $u\in C^{m_x+1,m_y+1}(\Omega_k)$. Then there exists a constant $C_k>0$, depending on $u$ and the degrees $m_x,m_y$, such that
\[
\|u-\mathcal I_k[u]\|_{\Omega_k}
\leq
C_k
\left(
h^{m_x+1}+h^{m_y+1}+h^{m_x+m_y+2}
\right),
\]
where $\|.\|_{\Omega_k}$ is the supremum norm on $\Omega_k$,
\[
h=\max\{h_{x,k}, h_{y,k}\},\quad h_{x,k}=\max_{1\leq r\leq n_x-1}|x_r^k-x_{r+1}^k|,
\quad\text{and}\quad
h_{y,k}=\max_{1\leq s\leq n_y-1}|y_s^k-y_{s+1}^k|.
\]
In particular, if $m=m_x=m_y$ and $h<1$, then
\[
\|u-\mathcal I_k[u]\|_{\Omega_k}
\leq
\mathcal C_kh^{m+1}.
\]
\end{proposition}

\begin{proof}
We denote by $I_x^{(k)}$ and $I_y^{(k)}$ the one-dimensional Lagrange interpolation operators associated with the $x$- and $y$-nodes of $\mathcal X_k$, respectively. Thus,
\[
\mathcal I_k=I_x^{(k)}I_y^{(k)} .
\]
Since $I_y^{(k)}$ acts only on the variable $y$, it commutes with differentiation with respect to $x$. Hence, for every integer $\textbf{r}\geq0$,
\[
\partial_x^\textbf{r}\bigl(I_y^{(k)}u\bigr)
=
I_y^{(k)}\bigl(\partial_x^\textbf{r} u\bigr).
\]

We decompose the interpolation error as
\[
u-\mathcal I_k[u]
=A_1+A_2,
\]
where \[A_1:=\bigl(u-I_y^{(k)}u\bigr)
\quad\text{and}\quad
A_2:=\bigl(I_y^{(k)}u-I_x^{(k)}I_y^{(k)}u\bigr).
\]

The term $A_1$ is the one-dimensional interpolation error in the $y$-direction. Therefore, by the standard one-dimensional Lagrange remainder estimate,
\[
\|A_1\|_{\Omega_k}
\leq
C_{y,k}h_{y,k}^{m_y+1}
\|\partial_y^{m_y+1}u\|_{\Omega_k}.
\]
For fixed $y$, $A_2$ is the interpolation error in the $x$-direction of the function $I_y^{(k)}u$. Hence
\[
\|A_2\|_{\Omega_k}
\leq
C_{x,k}h_{x,k}^{m_x+1}
\left\|
\partial_x^{m_x+1}\bigl(I_y^{(k)}u\bigr)
\right\|_{\Omega_k}.
\]
Using the commutation property with $\textbf{r}=m_x+1$, and setting
\[
v=\partial_x^{m_x+1}u,
\]
we get
\[
\partial_x^{m_x+1}\bigl(I_y^{(k)}u\bigr)
=
I_y^{(k)}v .
\]
Furthermore,
\[
\|I_y^{(k)}v\|_{\Omega_k}
\leq
\|v\|_{\Omega_k}
+
\|v-I_y^{(k)}v\|_{\Omega_k}.
\]
Applying again the one-dimensional interpolation estimate in the $y$-direction gives
\[
\|v-I_y^{(k)}v\|_{\Omega_k}
\leq
C_{y,k}h_{y,k}^{m_y+1}
\|\partial_y^{m_y+1}v\|_{\Omega_k}.
\]
Since
\[
\partial_y^{m_y+1}v
=
\partial_x^{m_x+1}\partial_y^{m_y+1}u,
\]
we obtain
\[
\|I_y^{(k)}v\|_{\Omega_k}
\leq
\|\partial_x^{m_x+1}u\|_{\Omega_k}
+
C_{y,k}h_{y,k}^{m_y+1}
\|\partial_x^{m_x+1}\partial_y^{m_y+1}u\|_{\Omega_k}.
\]
Consequently,
\[
\|A_2\|_{\Omega_k}
\leq
C_{x,k}h_{x,k}^{m_x+1}
\left(
\|\partial_x^{m_x+1}u\|_{\Omega_k}
+
C_{y,k}h_{y,k}^{m_y+1}
\|\partial_x^{m_x+1}\partial_y^{m_y+1}u\|_{\Omega_k}
\right).
\]
Combining the estimates for $A_1$ and $A_2$, and absorbing the constants into a single constant $C$, yields
\[
\|u-\mathcal I_k[u]\|_{\Omega_k}
\leq
C
\left(
h_{x,k}^{m_x+1}\|\partial_x^{m_x+1}u\|_{\Omega_k}
+
h_{y,k}^{m_y+1}\|\partial_y^{m_y+1}u\|_{\Omega_k}
+
h_{x,k}^{m_x+1}h_{y,k}^{m_y+1}
\|\partial_x^{m_x+1}\partial_y^{m_y+1}u\|_{\Omega_k}
\right).
\]
Hence, \[
\|u-\mathcal I_k[u]\|_{\Omega_k}
\leq
C_k
\left(
h^{m_x+1}+h^{m_y+1}+h^{m_x+m_y+2}
\right),
\]
If $m=m_x=m_y$ and $h<1$, 
\[
h^{2m+2}\leq h^{m+1},
\]
and therefore
\[
\|u-\mathcal I_k[u]\|_{\Omega_k}
\leq
\mathcal C_kh^{m+1}.
\]
\end{proof}

\subsection{Construction of Global Interpolant}
In order to combine the local interpolants into a global approximation operator, we associate with each subgrid $\mathcal{X}_k$ a nonnegative weight function
\[
\mathcal{W}_k:\mathbb{R}^2\to[0,1],
\]
satisfying the partition of unity condition
\begin{equation}
\sum_{k=1}^M \mathcal{W}_k(\mathbf{x})=1,
\qquad \mathbf{x}\in\mathbb{R}^2.
\label{eq:Punity_rew}
\end{equation}
The global operator is then defined by
\begin{equation}
\mathcal{G}[u,\mathcal{X}](\mathbf{x})
=
\sum_{k=1}^M \mathcal{W}_k(\mathbf{x})\,\mathcal{I}_k[u](\mathbf{x}).
\label{eq:Oper_g_rew}
\end{equation}

In the present framework, the weights are defined on the whole plane and are obtained by normalization of inverse-distance products, namely, for $\mathbf{x}\in\mathbb{R}^2$,
\begin{equation}
\mathcal{W}_k(\mathbf{x})
=
\frac{\omega_k(\mathbf{x})}
{ \sum_{l=1}^M \omega_l(\mathbf{x})},
\label{eq:weights_raw}
\end{equation}
 with
\[
\omega_k(\mathbf{x})=\prod_{\mathbf{x}'\in\mathcal{X}_k}\|\mathbf{x}-\mathbf{x}'\|^{-\mu},\qquad \mu>2.
\] 
By construction, the functions $\mathcal{W}_k$ satisfy the partition of unity property given by \ref{eq:Punity_rew}. Moreover, they enjoy cardinal-type properties.
\begin{itemize}
\item[$\mathrm{P1}$:] For every $\mathbf{x}_i\in\mathcal{X}\setminus\mathcal{X}_k$: \[\mathcal{W}_k(\mathbf{x}_i)=0,\quad \nabla \mathcal{W}_k(\mathbf{x}_i)=0,\quad\text{and}\quad H\mathcal{W}_k(\mathbf{x}_i)=0.\]
\item[$\mathrm{P2}$:] For every $\mathbf{x}_i\in\mathcal{X}$: \
\[
\sum_{k:\,\mathbf{x}_i\in\mathcal{X}_k}\mathcal{W}_k(\mathbf{x}_i)=1,\quad 
\sum_{k:\,\mathbf{x}_i\in\mathcal{X}_k}\nabla \mathcal{W}_k(\mathbf{x}_i)=0,\quad\text{and}\quad
\sum_{k:\,\mathbf{x}_i\in\mathcal{X}_k}H\mathcal{W}_k(\mathbf{x}_i)=0.
\]

where $\nabla$ is the gradient operator and $H$ is the Hessian operator.
\end{itemize}

For the numerical treatment of differential operators, it is convenient to rewrite $\mathcal{G}[u,\mathcal{X}]$ explicitly as a linear combination of the nodal values $u_{i,j}$. Substituting \ref{eq:Loca_Int_rew} into \ref{eq:Oper_g_rew} gives
\[
\mathcal{G}[u,\mathcal{X}](\mathbf{x})
=
\sum_{k=1}^M\sum_{r=1}^{n_x}\sum_{s=1}^{n_y}
u_{r,s}^k\,\Phi_{r,s}^k(\mathbf{x})\,\mathcal{W}_k(\mathbf{x}).
\]
Collecting equal nodal values, we obtain
\begin{equation}\label{eq:G_double_index}
\mathcal{G}[u,\mathcal{X}](\mathbf{x})
=\sum_{i=1}^{N_x}\sum_{j=1}^{N_y} u_{i,j}\,B_{i,j}(\mathbf{x}),
\end{equation}

where
\begin{equation}
B_{i,j}(\mathbf{x})
=
\sum_{{k'}\in K_{i,j}} \Phi_{i,j}^{k'}(\mathbf{x})\,\mathcal{W}_{{k'}}(\mathbf{x}),
\label{eqpointBase_rew}
\end{equation}
and $K_{i,j}$ denotes the set of indices ${k'}$ such that $\mathbf{x}_{i,j}\in\mathcal{X}_{{k'}}$.

For notational convenience, we replace this double index $(i,j)$ by a single index $\lambda$ by introducing the standard bijection
\[
\lambda=(i-1)N_y+j,
\qquad 1\le i\le N_x,\quad 1\le j\le N_y.
\]
The representation form given by Formula \ref{eq:G_double_index} can be written in single-index form as
\begin{equation}
\mathcal{G}[u,\mathcal{X}](\mathbf{x})
=
\sum_{\lambda=1}^{N} u_\lambda\,B_\lambda(\mathbf{x}),
\qquad N=N_x\times N_y.
\label{eqUwithsinglsum_rew}
\end{equation}
Here, 
\[
B_\lambda(\mathbf x)
=
\sum_{k\in K_\lambda}
\Phi_\lambda^{k}(\mathbf x)\mathcal W_k(\mathbf x),
\qquad
K_\lambda=\{k:\mathbf x_\lambda\in\mathcal X_k\}.
\]
The form \ref{eqUwithsinglsum_rew} is used in the discretization of the differential problems considered in the following.

The following proposition states the cardinal property of the global basis functions $B_i$ and the reconstruction property of the operator $\mathcal G$.
\begin{proposition}~~

\begin{itemize}
    \item The global basis functions $B_i$ satisfy the cardinal property
    \[B_i(\mathbf x_j)=\delta_{ij},\qquad i,j=1,\ldots,N.\]
    \item The operator $\mathcal G$ interpolates the nodal data
    \[\mathcal G[u,\mathcal X](\mathbf x_j)=u_j,\qquad j=1,\ldots,N.\]
\end{itemize}
\end{proposition}

\begin{proof}
At a node $\mathbf x_j$, only the subgrids containing $\mathbf x_j$ contribute, because
$\mathcal W_k(\mathbf x_j)=0$ otherwise. Hence
\[
B_i(\mathbf x_j)
=
\sum_{k:\,\mathbf x_i,\mathbf x_j\in\mathcal X_k}
\Phi_i^{(k)}(\mathbf x_j)\mathcal W_k(\mathbf x_j)
=
\delta_{ij}
\sum_{k:\,\mathbf x_j\in\mathcal X_k}\mathcal W_k(\mathbf x_j).
\]
The partition of unity gives
\[
\sum_{k:\,\mathbf x_j\in\mathcal X_k}\mathcal W_k(\mathbf x_j)=1,
\]
and therefore $B_i(\mathbf x_j)=\delta_{ij}$. The interpolation property follows from the representation
$\mathcal G[u,\mathcal X]=\sum_i u_iB_i$.
\end{proof}
The following proposition states the polynomial reproduction property of the GBMSC operator $\mathcal G$.
\begin{proposition}
Let $q\in\mathbb P_{m_x,m_y}$, where
\[
\mathbb P_{m_x,m_y}
=\mathbb P_{m_x}\otimes \mathbb P_{m_y}=
\operatorname{span}\{x^\alpha y^\beta:\ 0\le \alpha\le m_x,\ 0\le \beta\le m_y\}.
\]
For each subgrid $\mathcal X_k$ of size $n_x\times n_y$, with $n_x=m_x+1$ and $n_y=m_y+1$, we have
\[
\mathcal I_k[q](\mathbf x)=q(\mathbf x),
\qquad k=1,\ldots,M.
\]
Hence,
\[
\mathcal G[q,\mathcal X](\mathbf x)=q(\mathbf x),
\qquad \mathbf x\in\mathcal D.
\]
\end{proposition}
\begin{proof}
Let $q\in\mathbb P_{m_x,m_y}$. By definition, $q$ can be written as
\[
q(x,y)=\sum_{\alpha=0}^{m_x}\sum_{\beta=0}^{m_y}
c_{\alpha\beta}x^\alpha y^\beta .
\]
The local interpolant $\mathcal I_k[q]$ is the tensor product of the one-dimensional Lagrange interpolation operators on the $x$- and $y$-nodes of $\mathcal X_k$. Since the one-dimensional interpolation in the $x$-direction is exact on $\mathbb P_{m_x}$, and the one-dimensional interpolation in the $y$-direction is exact on $\mathbb P_{m_y}$, their tensor product is exact on every monomial $x^\alpha y^\beta$, with
\[
0\leq \alpha\leq m_x,\qquad 0\leq \beta\leq m_y.
\]
By linearity,
\[
\mathcal I_k[q](\mathbf x)=q(\mathbf x),
\qquad k=1,\ldots,M.
\]
Therefore, using the definition of the operator $\mathcal G$ and the partition of unity property,
\[
\mathcal G[q,\mathcal X](\mathbf x)
=
\sum_{k=1}^M \mathcal W_k(\mathbf x)\mathcal I_k[q](\mathbf x)
=
q(\mathbf x)\sum_{k=1}^M\mathcal W_k(\mathbf x)
=
q(\mathbf x).
\]
\end{proof}
The next proposition establishes the consistency of the differentiated basis functions of the operator $\mathcal{G}$.
\begin{proposition}
For every
$q\in\mathbb P_{m_x,m_y}$, we have
\[
\partial_x\mathcal G[q,\mathcal X]=\partial_x q,\qquad
\partial_y\mathcal G[q,\mathcal X]=\partial_y q,
\]
and
\[
\partial_{xx}\mathcal G[q,\mathcal X]=\partial_{xx}q,\qquad
\partial_{xy}\mathcal G[q,\mathcal X]=\partial_{xy}q,\qquad
\partial_{yy}\mathcal G[q,\mathcal X]=\partial_{yy}q.
\]

\end{proposition}
\begin{proof}
Let $q\in\mathbb P_{m_x,m_y}$. By the polynomial reproduction property of the local tensor-product interpolants,
\[
\mathcal I_k[q](\mathbf x)=q(\mathbf x),
\qquad k=1,\ldots,M.
\]
Hence, by the definition of the GBMSC operator,
\[
\mathcal G[q,\mathcal X](\mathbf x)
=
\sum_{k=1}^M \mathcal W_k(\mathbf x)\mathcal I_k[q](\mathbf x)
=
q(\mathbf x)\sum_{k=1}^M\mathcal W_k(\mathbf x)
=
q(\mathbf x).
\]
Thus, $\mathcal G[q,\mathcal X]$ and $q$ coincide at every point where the weights are defined. Since the weights are differentiable twice at $\mathbf x$, the function $\mathcal G[q,\mathcal X]$ is twice differentiable at $\mathbf x$. Differentiating the identity
\[
\mathcal G[q,\mathcal X]=q
\]
completes the proof.
\end{proof}

\section{Application to the Solution of Linear Elliptic PDEs}
In order to highlight the application of the GBMSC method, in this section we deal with linear partial differential equations. We consider steady-state elliptic problems. Within this category, we have selected four classical benchmarks, each endowed with an analytical solution, so that we can rigorously evaluate convergence rates. In the sections that follow, we describe these test cases in detail and present our numerical results to illustrate the capacity of the GBMSC framework to achieve high levels of accuracy.

We are interested in constructing an approximation of the solution to the following elliptic problem:
\begin{equation}
    (P_s): \begin{cases}
        \mathcal{L}u &= f \quad \text{in } \mathcal{D}, \\
        \mathcal{B}u &= g \quad \text{on } \partial\mathcal{D},
    \end{cases} 
    \label{eq:PbLimites} 
\end{equation}
where $\mathcal{D}$ is a bounded rectangular spatial domain, $\mathcal{L}$ is an elliptic linear differential operator and $\mathcal{B}$ is a linear operator describing the boundary conditions.

The exact solution $u$ of ($P_s$) is approximated using the GBMSC method by Formula \ref{eqUwithsinglsum_rew} as 
\begin{equation}
    u(\mathbf{x})\approx \widetilde{u}(\mathbf{x})
=
\sum_{j=1}^{N}\widetilde{u}_j\,B_j(\mathbf{x}).\label{eq:approsol}
\end{equation}

Substituting the approximate solution \ref{eq:approsol} into the boundary value problem \ref{eq:PbLimites} and enforcing the equation at all collocation nodes yields
\begin{equation}
   \begin{cases}
\mathcal{L}\widetilde{u}(\mathbf{x}_i)=f(\mathbf{x}_i), & i\in I,\\
\mathcal{B}\widetilde{u}(\mathbf{x}_i)=g(\mathbf{x}_i), & i\in B,
\end{cases} \label{eq:coloPbLimites} 
\end{equation}
where $I$ and $B$ denote the index sets for the interior and boundary nodes, respectively, defined by:
\[I =\{ k\in \{1,\cdots,N\}: \mathbf{x}_k\in\mathcal D\}, \text{ and }  
B=\{k\in \{1,\cdots,N\}: \mathbf{x}_k\in\partial\mathcal D\}.\]

Since the operators $\mathcal{L}$ and $\mathcal{B}$ are linear, the system of linear equations \ref{eq:coloPbLimites} can be  rewritten in matrix form as \[\mathcal{A}\widetilde{\mathbf{U}} = \mathcal{F},\] where  $\widetilde{\mathbf{U}} = [\widetilde{u}_1, \dots, \widetilde{u}_N]^T$ is the vector of unknown and  
\[
\mathcal{A}
=
\begin{bmatrix}
[\mathcal{L}B_j(\mathbf{x}_i)]_{i\in I,\; j=1,\dots,N}\\[4pt]
[\mathcal{B}B_j(\mathbf{x}_i)]_{i\in B,\; j=1,\dots,N}
\end{bmatrix},
\qquad
\mathcal{F}
=
\begin{bmatrix}
[f(\mathbf{x}_i)]_{i\in I}\\[2pt]
[g(\mathbf{x}_i)]_{i\in B}
\end{bmatrix}.
\]

In this paper, all proposed elliptic benchmark problems are treated in the rectangular domain $\mathcal{D}$. This domain is discretized by a uniform Cartesian grid of $N=N_x \times N_y$ nodes,  where these grid nodes serve as the collocation points $\{\mathbf{x}_i\}_{i=1}^{N}$.  
To ensure connectivity between local subgrids, we construct subgrids $\mathcal{X}_k$ of size $ n_x \times n_y$ that satisfy the overlapping conditions $\eta_x$ and $\eta_y$. The parameters $\eta_x$ and $\eta_y$ control the interaction between neighboring local interpolants and play an important role in the compromise between locality and stability.

For the numerical solution of elliptic problems, the derivatives of the basis functions $B_j$ are not assembled by differentiating the global representation directly, but by exploiting the local tensor-product structure of the interpolants and the derivatives of the partition functions. This yields a sparse discrete operator and makes the implementation efficient. We set $\mu = 4.01$ to strictly control and stabilize the behavior of $B_j$. The multinode Shepard exponent used in all numerical experiments is \(\mu=4.01\). This corrects the value reported in the first version. The numerical results are unchanged.

A delicate aspect in the construction of the weights $\mathcal{W}_k$ is their numerical evaluation. The formula \ref{eq:weights_raw} can produce very large or very small intermediate quantities. In the implementation, this issue is handled by a min-max logarithmic stabilization of the raw weights.

More precisely, we first compute
\[
\lambda_k(\mathbf{x})=\log \omega_k(\mathbf{x}),
\]
and then shift the logarithms with respect to the maximum value
\[
\lambda_{\max}(\mathbf{x})=\max_{1\le k\le M}\lambda_k(\mathbf{x}).
\]
If necessary, the admissible logarithmic range is further truncated by a min--max rule before exponentiation. The stabilized quantities are therefore reconstructed as
\[
\widetilde{\omega}_k(\mathbf{x})
=
\exp\bigl(\widetilde{\lambda}_k(\mathbf{x})-\lambda_{\max}(\mathbf{x})\bigr),
\]
where $\widetilde{\lambda}_k$ denotes the logarithmic value of the clipped value, and the final normalized weights are given by
\[
\mathcal{W}_k(\mathbf{x})
=
\frac{\widetilde{\omega}_k(\mathbf{x})}
{\sum_{l=1}^M \widetilde{\omega}_l(\mathbf{x})}.
\]
This procedure avoids overflow and underflow, reduces the effect of extreme scale separation among the raw weights, and yields a more regular partition of unity. In particular, it prevents an excessive concentration of the weighting on a single subgrid, as well as an excessively uniform redistribution over too many subgrids. From a numerical point of view, this balance is important in order to avoid both overlocalized and overdiffused behavior of the global approximant.

Let $\mathcal{X}_k$ be a local subgrid of size $n_x\times n_y$, let \[
\mathcal J_k=\{(r,s):\ r=1,\dots,n_x,\ \ s=1,\dots,n_y\},
\] denote the set of pairs of local indices of $\mathcal{X}_k$, and let $(\alpha,\beta)$ and $(r,s)$ be two pairs of local indices $\mathcal J_k$, respectively, corresponding to an evaluation node and a source node. We define the one-dimensional differentiation matrices
\[
(L_{q,x}^{(k)})_{\alpha r}
=
\bigl(\ell_{x,r}^{k}\bigr)^{(q)}(x_{\alpha}),
\qquad
(L_{q,y}^{(k)})_{\beta s}
=
\bigl(\ell_{y,s}^{k}\bigr)^{(q)}(y_{\beta}),
\qquad q=0,1,2.
\]
In particular, $L_{0,x}^{(k)}$ and $L_{0,y}^{(k)}$ are the identity matrices. Moreover, for the weight function $\mathcal{W}_k$, we denote its values and derivatives at the local evaluation node $(x_{\alpha},y_{\beta})$ by
\[
\mathbf{W}^{(k)}_{\alpha\beta}= \mathcal{W}_{k}(x_\alpha,y_\beta),
\qquad
\mathbf{W}^{(k)}_{x,\alpha\beta}=\mathcal{W}_{x,k}(x_\alpha,y_\beta),
\qquad
\mathbf{W}^{(k)}_{y,\alpha\beta}=\mathcal{W}_{y,k}(x_\alpha,y_\beta),\]
\[
\mathbf{W}^{(k)}_{xx,\alpha\beta}=\mathcal{W}_{xx,k}(x_\alpha,y_\beta),
\qquad
\mathbf{W}^{(k)}_{yy,\alpha\beta}=\mathcal{W}_{yy,k}(x_\alpha,y_\beta).
\]

The local contributions to the first- and second-order differentiation matrices are then
\[
\bigl(\mathcal M_x^{(k)}\bigr)_{(\alpha,\beta),(r,s)}
=
\delta_{\beta s}
\left(
\mathbf{W}^{(k)}_{x,\alpha\beta}(L_{0,x}^{(k)})_{\alpha r}
+
\mathbf{W}^{(k)}_{\alpha\beta}(L_{1,x}^{(k)})_{\alpha r}
\right),
\]
\[
\bigl(\mathcal M_{xx}^{(k)}\bigr)_{(\alpha,\beta),(r,s)}
=
\delta_{\beta s}
\left(
\mathbf{W}^{(k)}_{xx,\alpha\beta}(L_{0,x}^{(k)})_{\alpha r}
+
2\mathbf{W}^{(k)}_{x,\alpha\beta}(L_{1,x}^{(k)})_{\alpha r}
+
\mathbf{W}^{(k)}_{\alpha\beta}(L_{2,x}^{(k)})_{\alpha r}
\right),
\]
\[
\bigl(\mathcal M_y^{(k)}\bigr)_{(\alpha,\beta),(r,s)}
=
\delta_{\alpha r}
\left(
\mathbf{W}^{(k)}_{y,\alpha\beta}(L_{0,y}^{(k)})_{\beta s}
+
\mathbf{W}^{(k)}_{\alpha\beta}(L_{1,y}^{(k)})_{\beta s}
\right),
\]
\[
\bigl(\mathcal M_{yy}^{(k)}\bigr)_{(\alpha,\beta),(r,s)}
=
\delta_{\alpha r}
\left(
\mathbf{W}^{(k)}_{yy,\alpha\beta}(L_{0,y}^{(k)})_{\beta s}
+
2\mathbf{W}^{(k)}_{y,\alpha\beta}(L_{1,y}^{(k)})_{\beta s}
+
\mathbf{W}^{(k)}_{\alpha\beta}(L_{2,y}^{(k)})_{\beta s}
\right),
\]
 where $\delta$ is the Kronecker delta. These formulas are the discrete equivalent of the product rule and explicitly show how the derivatives of the local interpolants interact with the derivatives of the weighting functions.

The corresponding global sparse matrices are obtained by summing all local contributions associated with subgrids containing the pair of global nodes under consideration. Thus, if $i$ and $j$ are global indices,
\[
(\mathcal M_{\star})_{ij}
=
\sum_{k:\, \pi_k(i),\pi_k(j)\in \mathcal J_k}
\bigl(\mathcal M_{\star}^{(k)}\bigr)_{\pi_k(i),\pi_k(j)},
\qquad
\star\in\{x,xx,y,yy\},
\]
where $\pi_k(i)$ denotes the pair of the local index in $\mathcal J_k$ of the global node $i$.

The matrices $\mathcal M_x$, $\mathcal M_y$, $\mathcal M_{xx}$ and $\mathcal M_{yy}$ are nodal differential operators acting on the vector of unknown nodal coefficients
\[
\widetilde{\mathbf U}
=
(\widetilde u_1,\ldots,\widetilde u_N)^T .
\]
More precisely, their action provides the discrete approximations
\[
\mathcal M_x\widetilde{\mathbf U}\simeq
\bigl(\partial_x \widetilde u(\mathbf{x}_i)\bigr)_{i=1}^N,
\qquad
\mathcal M_y\widetilde{\mathbf U}\simeq
\bigl(\partial_y \widetilde u(\mathbf{x}_i)\bigr)_{i=1}^N,
\]
and
\[
\mathcal M_{xx}\widetilde{\mathbf U}\simeq
\bigl(\partial_{xx} \widetilde u(\mathbf{x}_i)\bigr)_{i=1}^N,
\qquad
\mathcal M_{yy}\widetilde{\mathbf U}\simeq
\bigl(\partial_{yy} \widetilde u(\mathbf{x}_i)\bigr)_{i=1}^N .
\]
Since the rows are evaluated at grid nodes and each row involves only the local subgrids containing the corresponding node, the resulting matrices are sparse. This sparsity is a direct consequence of the local tensor-product Lagrange structure and of the compact interaction induced by the selected overlapping subgrids.

The same discrete operators can be used to assemble more general linear elliptic equations of the form
\begin{equation}
\nabla\cdot(D\nabla u)+\mathbf v\cdot\nabla u+c\,u=f
\qquad\hbox{in }\mathcal D,
\label{eq:general-linear-elliptic-added}
\end{equation}
where
\[
D(\mathbf{x})
=
\begin{pmatrix}
D_x(\mathbf{x}) & 0\\
0 & D_y(\mathbf{x})
\end{pmatrix},
\qquad
\mathbf v(\mathbf{x})=(v_x(\mathbf{x}),v_y(\mathbf{x})),
\]
and $c$ is a reaction coefficient. By the product rule,
\[
\nabla\cdot(D\nabla u)
=
D_x\,\partial_{xx}u
+
(\partial_xD_x)\,\partial_xu
+
D_y\,\partial_{yy}u
+
(\partial_yD_y)\,\partial_yu .
\]
Let
\[
(\mathbf d_x)_i=D_x(\mathbf{x}_i),\qquad
(\mathbf d_y)_i=D_y(\mathbf{x}_i),\qquad
(\mathbf v_x)_i=v_x(\mathbf{x}_i),\qquad
(\mathbf v_y)_i=v_y(\mathbf{x}_i),\qquad
\mathbf c_i=c(\mathbf{x}_i).
\]
Approximating the derivatives of the diffusion coefficients by the same nodal operators gives the interior matrix
\[
\mathcal A_D
=
\operatorname{diag}(\mathbf d_x)\mathcal M_{xx}
+
\operatorname{diag}(\mathcal M_x\mathbf d_x)\mathcal M_x
+
\operatorname{diag}(\mathbf d_y)\mathcal M_{yy}
+
\operatorname{diag}(\mathcal M_y\mathbf d_y)\mathcal M_y .
\]
The convection and reaction contributions are
\[
\mathcal A_{\mathbf v}
=
\operatorname{diag}(\mathbf v_x)\mathcal M_x
+
\operatorname{diag}(\mathbf v_y)\mathcal M_y,
\qquad
\mathcal A_c=\operatorname{diag}(\mathbf c).
\]
Therefore, before imposing the boundary conditions, the discrete operator associated with \ref{eq:general-linear-elliptic-added} is
\[
\mathcal A_{\mathrm{int}}
=
\mathcal A_D+\mathcal A_{\mathbf v}+\mathcal A_c .
\]

Let
\[
\partial\mathcal D=\Gamma_{\mathrm{Dir}}\cup\Gamma_{\mathrm{Neu}}\cup\Gamma_{\mathrm{Rob}},
\qquad
\Gamma_{\mathrm{Dir}}\cap\Gamma_{\mathrm{Neu}}=\Gamma_{\mathrm{Dir}}\cap\Gamma_{\mathrm{Rob}}=\Gamma_{\mathrm{Neu}}\cap\Gamma_{\mathrm{Rob}}=\emptyset,
\]
where Dirichlet, Neumann, and Robin conditions are prescribed, respectively, in the form
\[
u=g_{\mathrm{Dir}} \quad \hbox{on } \Gamma_{\mathrm{Dir}},
\]
\[
\frac{\partial u}{\partial\mathbf n}
=
\nabla u\cdot \mathbf n
=
g_{\mathrm{Neu}}
\quad \hbox{on } \Gamma_{\mathrm{Neu}},
\]
and
\[
\alpha u
+
\beta\frac{\partial u}{\partial\mathbf n}
=
g_{\mathrm{Rob}}
\quad \hbox{on } \Gamma_{\mathrm{Rob}} .
\]
Here $\mathbf n=(n_x,n_y)^T$ denotes the outward unit normal to
$\partial\Omega$, while $\alpha$ and $\beta$ are prescribed boundary
coefficients.

Boundary conditions are imposed by replacing the corresponding rows of
the algebraic system. If $\mathbf x_i\in\Gamma_{\mathrm{Dir}}$, then the $i$-th row is
replaced by the coordinate row $\mathbf e_i^T$, where $\mathbf e_i$ is
the $i$-th vector of the canonical basis of $\mathbb R^N$, and the
right-hand side is set equal to $g_{\mathrm{Dir}}(\mathbf x_i)$. If
$\mathbf x_i\in\Gamma_{\mathrm{Neu}}$, with
$\mathbf n_i=(n_{x,i},n_{y,i})^T$, then the corresponding row is
\[
n_{x,i}(\mathcal M_x)_{i,:}
+
n_{y,i}(\mathcal M_y)_{i,:},
\]
and the right-hand side is $g_{\mathrm{Neu}}(\mathbf x_i)$. Thus the normal
derivative is discretized by the same nodal first-derivative matrices
used in the interior operator.

Finally, if $\mathbf x_i\in\Gamma_{\mathrm{Rob}}$, the Robin condition is imposed by
the row
\[
\alpha_i\mathbf e_i^T
+
\beta_i
\left(
n_{x,i}(\mathcal M_x)_{i,:}
+
n_{y,i}(\mathcal M_y)_{i,:}
\right),
\]
with right-hand side $g_{\mathrm{Rob}}(\mathbf x_i)$, where
\[
\alpha_i=\alpha(\mathbf x_i),
\qquad
\beta_i=\beta(\mathbf x_i).
\]
In the numerical tests below we use only Dirichlet and Neumann rows.

For the Poisson operator with constant diffusion,
\[
\Delta u=f
\qquad \text{in }\mathcal{D},
\]
and homogeneous Dirichlet data, the discrete stiffness matrix reduces, before the imposition of the boundary conditions, to
\[
\mathcal{A}=\mathcal M_{xx}+\mathcal M_{yy}.
\]
The boundary conditions are then enforced by replacing each boundary row with the corresponding row of the identity matrix. This operation preserves the sparse structure of the algebraic system.

The implementation is highly sparse, Algorithm~\ref{algo:Lapla}. The local contributions are first accumulated as triplets in COO format, then converted into CSR format for algebraic operations and for the solution step. Only the imposition of Dirichlet conditions requires a temporary conversion to the LIL format, which allows an efficient modification of the boundary rows. This strategy keeps both the storage cost and the assembly cost under control.

\begin{algorithm}
\caption{Sparse assembly of the Poisson stiffness matrix $\mathcal{A}$}
\label{algo:Lapla}
\begin{algorithmic}[1]
\Require Structured grid of size $N=N_x\times N_y$ and sampled subgrids $\{\mathcal{X}_k\}_{k=1}^{M}$ of size $n_x\times n_y$
\State Initialize the triplet lists $\mathcal{T}_{xx}$ and $\mathcal{T}_{yy}$
\For{$k=1,\dots,M$}
    \State Extract the global indices $(\mathbf{i}_1^{(k)},\dots,\mathbf{i}_{n_x\times n_y}^{(k)})$ of the pair indices of the nodes of $\mathcal{X}_k$
    \State Compute $L_{q,x}^{(k)}$ and $L_{q,y}^{(k)}$, $q=0,1,2$
    \State Evaluate the stabilized Shepard quantities $\mathbf{W}^{(k)}$, $\mathbf{W}_x^{(k)}$, $\mathbf{W}_y^{(k)}$, $\mathbf{W}_{xx}^{(k)}$, and $\mathbf{W}_{yy}^{(k)}$ at the nodes of $\mathcal{X}_k$
    \For{each local pair $((\alpha,\beta), (r,s))$ in $\mathcal{J}_k$}
        \State Compute $\bigl(\mathcal M_{xx}^{(k)}\bigr)_{(\alpha,\beta),(r,s)}$ and $\bigl(\mathcal M_{yy}^{(k)}\bigr)_{(\alpha,\beta),(r,s)}$
        \State Append $\left(\mathbf{i}_{\alpha\beta}^{(k)}, \mathbf{i}_{rs}^{(k)}, \bigl(\mathcal M_{xx}^{(k)}\bigr)_{(\alpha,\beta),(r,s)}\right)$ to $\mathcal{T}_{xx}$
        \State Append $\left(\mathbf{i}_{\alpha\beta}^{(k)}, \mathbf{i}_{rs}^{(k)}, \bigl(\mathcal M_{yy}^{(k)}\bigr)_{(\alpha,\beta),(r,s)}\right)$ to $\mathcal{T}_{yy}$
    \EndFor
\EndFor
\State Build $\mathcal M_{xx}$ and $\mathcal M_{yy}$ in COO format from $\mathcal{T}_{xx}$ and $\mathcal{T}_{yy}$
\State Convert $\mathcal M_{xx}$ and $\mathcal M_{yy}$ to CSR format
\State Form $\mathcal{A}=\mathcal M_{xx}+\mathcal M_{yy}$
\State Convert $\mathcal{A}$ to LIL format, replace each Dirichlet boundary row by the corresponding identity row, and convert back to CSR
\State \Return $\mathcal{A}$
\end{algorithmic}
\end{algorithm}

To rigorously quantify the accuracy of our method, the error is evaluated on an independent, highly refined Cartesian grid of $Z = 101 \times 101$ points, denoted as $\{\boldsymbol{\xi}_j\}_{j=1}^{Z}$. We define the pointwise error at these evaluation points as $e(\boldsymbol{\xi}_j) := u(\boldsymbol{\xi}_j) - \widetilde{u}(\boldsymbol{\xi}_j)$. We compute the mean absolute error ($e_{\text{mean}}$), the root mean square error ($e_{\text{RMS}}$), and the maximum absolute error ($e_{\text{max}}$) as follows:

\[e_{\text{mean}} = \frac{1}{Z} \sum_{j=1}^{Z} |e(\boldsymbol{\xi}_j)|, \quad e_{\text{RMS}} = \sqrt{\frac{1}{Z} \sum_{j=1}^{Z} |e(\boldsymbol{\xi}_j)|^2}, \quad \text{and} \quad e_{\text{max}} = \max_{1 \leq j \leq Z} |e(\boldsymbol{\xi}_j)|.\]

All computational implementations are written in Python, with \texttt{NumPy} library for vectorized array operations and \texttt{SciPy} library for sparse matrix assembly. The obtained linear system, $\mathcal{A}\widetilde{\mathbf{U}} = \mathcal{F}$, is resolved in double precision (\texttt{float64}) using a sparse direct solver (\texttt{scipy.sparse.linalg.spsolve}). The computational efficiency is assessed by wall-clock elapsed time. All computations are performed with \texttt{gbmsc-pde}~\cite{el_harrak_2026_gbmsc_pde},
a Python implementation of the Grid-Based Multinode Shepard Collocation method.

For each configuration, we report "CPU (s)" as the mean execution time computed over five independent runs. All numerical simulations were performed on a workstation equipped with an Intel Core i5-10600 processor operating at 3.10 GHz with 16 GB of RAM.

\subsection{Poisson Equation with Homogeneous Dirichlet Conditions}
We consider the following Poisson equation on the unit square domain $\mathcal{D} = (0, 1)^2$:
\begin{equation}\label{Eq.P1_elliptic_Pr}
    \left\{\begin{array}{ll}
         -\Delta u = f& \text{in }\mathcal{D}=(0,1)^2,\\
u=0& \text{on }\partial\mathcal{D}.  
         
    \end{array}\right.
\end{equation}

We perform the numerical validation using two exact solutions $u_1$ (see \cite{cangiani2017posteriori}) and $u_2$ (see \cite{mantzaflaris2015integration}) to Problem \ref{Eq.P1_elliptic_Pr}:
\begin{equation}\label{exactsolutions}
u_1(x, y) = 16x(1 - x)y(1 - y) \quad \text{and}\quad u_2(x, y) = \sin(\pi x)\,\sin(\pi y),
\end{equation}
which yields the two corresponding source terms $f_1$ and $f_2$ respectively:
\[
f_1(x, y) =32\left(x(1 - x)+y(1 - y)\right) \quad \text{and}\quad  f_2(x, y) =2\pi^2\sin(\pi x)\,\sin(\pi y).
\]
In this section, we take $N_x=N_y$ and $n_x=n_y$. In local subgrids $\mathcal X_k$, the GBMSC discretization uses local polynomials of degree $m$ in each variable, where $m=n_x-1$.

\subsubsection{Polynomial Reproduction}
To verify that the GBMSC method correctly reproduces polynomials, we take the Poisson problem \ref{Eq.P1_elliptic_Pr} with $u_1$ as the exact solution, a case where the method is expected to recover the solution without any approximation error. Table \ref{Tab:repreductionPoly} reports errors $e_{\max}$, $e_{\mathrm{mean}}$, and $e_{\mathrm{RMS}}$ together with condition number $\kappa$, for several resolutions of the grid.

\begin{table}[h]
\caption{Numerical results for Problem \ref{Eq.P1_elliptic_Pr} : errors and condition number $\kappa$ for the exact solution $u_1$ at various grid sizes $N_x$.}\label{Tab:repreductionPoly}
\centering
\begin{tabular}{ccrrrrr}
\toprule
$m$ & $\eta$ & $N_x$ & $e_{\text{max}}$ & $e_{\text{mean}}$ & $e_{\text{RMS}}$ & $\kappa$ \\
\midrule
\multirow{2}{*}{1} & \multirow{2}{*}{1} & 6 & 8.47e-01 & 3.68e-01 & 4.45e-01 & 1.56e+03 \\
 &  & 12 & 8.36e-01 & 3.64e-01 & 4.41e-01 & 1.10e+04 \\
\midrule
\midrule
\multirow{4}{*}{2} & \multirow{2}{*}{2} & 5 & 1.67e-15 & 2.19e-16 & 3.50e-16 & 1.55e+02 \\
 &  & 7 & 1.78e-15 & 4.24e-16 & 5.82e-16 & 4.37e+02 \\
\cmidrule{2-7}
 & \multirow{2}{*}{1} & 5 & 2.00e-15 & 3.38e-16 & 5.19e-16 & 2.03e+03 \\
 &  & 7 & 1.89e-15 & 3.61e-16 & 5.18e-16 & 6.62e+03 \\
\midrule
\midrule
\multirow{4}{*}{4} & \multirow{2}{*}{4} & 8 & 1.11e-15 & 2.16e-16 & 2.83e-16 & 7.76e+02 \\
 &  & 11 & 1.54e-14 & 4.36e-15 & 5.89e-15 & 1.96e+03 \\
\cmidrule{2-7}
 & \multirow{2}{*}{2} & 8 & 8.88e-16 & 1.99e-16 & 2.46e-16 & 5.83e+02 \\
 &  & 11 & 4.00e-15 & 1.03e-15 & 1.34e-15 & 1.41e+03 \\
\bottomrule
\end{tabular}
\end{table}

Table~\ref{Tab:repreductionPoly} shows the numerical accuracy based on values of $m$. For $m=1$, the method fails to reproduce the solution $u_1$ exactly, which is expected since a polynomial of degree 1 is insufficient to represent a solution of degree 2. However, once $m$ is increased to 2, errors drop to near-machine precision, regardless of the grid resolution or the overlap parameter $\eta$, which confirms that the GBMSC method correctly reproduces polynomials of degree up to $m$. The same behavior continues to the higher degree $m=4$, where the errors remain at the level of rounding errors, typically of order $10^{-15}$, and the condition numbers remain well within manageable limits, ranging from $10^2$ to $10^4$. It is also crucial to note that, as the grid is progressively refined, the condition number increases, but only moderately.

\subsubsection{Convergence Study}
The performance and stability of the GBMSC method are evaluated by solving Problem \ref{Eq.P1_elliptic_Pr} in $\mathcal{D}$, with the exact solution $u_2$ defined in \ref{exactsolutions}. The numerical tests are designed to examine the effect of the local degree $m$, the number of points on the grid $N_x=N_y$ in each direction, and the overlap parameter $\eta$, which together determine the precision, conditioning, and computational cost of the discretization.

\begin{table}[h]
\centering
\caption{Numerical results for Problem \ref{Eq.P1_elliptic_Pr} using exact solution $u_2$ for different values of the local degree $m$, ${m \in\{4,6,8\}}$.}\label{tab:scaling_by_rho}
\small
\begin{tabular}{cccrrrrrrr}
\toprule
$m$ & $\eta$ & $N_x$ & $e_{\text{max}}$ & $e_{\text{mean}}$ & $e_{\text{RMS}}$ & $\kappa$ & $\mathrm{CPU}$(s) & Sparsity (\%) \\
\midrule
\multirow{12}{*}{4} & \multirow{4}{*}{4} & 29 & 1.56e-06 & 7.29e-07 & 8.04e-07 & 2.54e+04 & 0.805 & 98.32 \\
 &  & 41 & 2.63e-07 & 9.88e-08 & 1.18e-07 & 6.16e+04 & 2.131 & 99.11 \\
 &  & 53 & 7.18e-08 & 2.55e-08 & 3.03e-08 & 1.18e+05 & 3.724 & 99.45 \\
 &  & 65 & 2.61e-08 & 9.80e-09 & 1.16e-08 & 1.98e+05 & 5.817 & 99.63 \\
\cmidrule{2-9}
 & \multirow{4}{*}{2} & 29 & 4.95e-06 & 1.74e-06 & 2.23e-06 & 1.80e+04 & 0.066 & 98.69 \\
 &  & 41 & 1.19e-06 & 4.32e-07 & 5.53e-07 & 4.36e+04 & 0.111 & 99.31 \\
 &  & 53 & 4.15e-07 & 1.55e-07 & 1.97e-07 & 8.38e+04 & 0.303 & 99.58 \\
 &  & 65 & 1.80e-07 & 6.84e-08 & 8.70e-08 & 1.40e+05 & 0.591 & 99.71 \\
\cmidrule{2-9}
 & \multirow{4}{*}{1} & 29 & 3.38e-05 & 1.36e-05 & 1.71e-05 & 1.62e+06 & 0.057 & 98.87 \\
 &  & 41 & 8.42e-06 & 3.29e-06 & 4.11e-06 & 4.02e+06 & 0.095 & 99.41 \\
 &  & 53 & 2.89e-06 & 1.15e-06 & 1.44e-06 & 7.82e+06 & 0.218 & 99.64 \\
 &  & 65 & 1.28e-06 & 5.03e-07 & 6.28e-07 & 1.32e+07 & 0.457 & 99.76 \\
\midrule
\midrule
\multirow{12}{*}{6} & \multirow{4}{*}{6} & 31 & 1.06e-08 & 7.50e-09 & 7.80e-09 & 3.40e+04 & 6.287 & 97.90 \\
 &  & 43 & 1.02e-09 & 7.47e-10 & 7.78e-10 & 7.81e+04 & 13.111 & 98.84 \\
 &  & 55 & 1.79e-10 & 1.33e-10 & 1.39e-10 & 1.46e+05 & 23.328 & 99.26 \\
 &  & 67 & 4.50e-11 & 3.38e-11 & 3.53e-11 & 2.40e+05 & 35.715 & 99.49 \\
\cmidrule{2-9}
 & \multirow{4}{*}{3} & 31 & 5.91e-08 & 2.75e-08 & 3.20e-08 & 9.78e+05 & 0.126 & 98.35 \\
 &  & 43 & 7.82e-09 & 3.43e-09 & 4.05e-09 & 2.31e+06 & 0.284 & 99.10 \\
 &  & 55 & 1.67e-09 & 7.29e-10 & 8.70e-10 & 4.35e+06 & 0.484 & 99.43 \\
 &  & 67 & 5.01e-10 & 2.13e-10 & 2.56e-10 & 7.21e+06 & 1.016 & 99.61 \\
\cmidrule{2-9}
 & \multirow{4}{*}{1} & 31 & 2.92e-07 & 1.20e-07 & 1.50e-07 & 1.30e+07 & 0.073 & 98.65 \\
 &  & 43 & 3.98e-08 & 1.61e-08 & 2.01e-08 & 3.15e+07 & 0.126 & 99.27 \\
 &  & 55 & 8.89e-09 & 3.57e-09 & 4.46e-09 & 6.02e+07 & 0.337 & 99.54 \\
 &  & 67 & 2.68e-09 & 1.07e-09 & 1.34e-09 & 1.00e+08 & 0.634 & 99.69 \\
\midrule
\midrule
\multirow{12}{*}{8} & \multirow{4}{*}{8} & 25 & 7.26e-10 & 5.04e-10 & 5.24e-10 & 3.56e+04 & 12.192 & 96.17 \\
 &  & 33 & 5.82e-11 & 4.00e-11 & 4.15e-11 & 7.08e+04 & 27.683 & 97.61 \\
 &  & 49 & 1.53e-12 & 9.98e-13 & 1.04e-12 & 1.91e+05 & 70.105 & 98.82 \\
 &  & 57 & 3.90e-13 & 1.70e-13 & 1.89e-13 & 2.80e+05 & 96.850 & 99.11 \\
\cmidrule{2-9}
 & \multirow{4}{*}{5} & 25 & 6.54e-10 & 2.48e-10 & 2.96e-10 & 4.20e+05 & 0.161 & 96.73 \\
 &  & 33 & 5.25e-11 & 1.88e-11 & 2.22e-11 & 8.70e+05 & 0.370 & 97.98 \\
 &  & 49 & 1.45e-12 & 5.87e-13 & 7.01e-13 & 2.41e+06 & 0.992 & 99.02 \\
 &  & 57 & 9.85e-13 & 3.17e-13 & 4.23e-13 & 3.55e+06 & 1.515 & 99.26 \\
\cmidrule{2-9}
 & \multirow{4}{*}{1} & 25 & 1.76e-08 & 7.60e-09 & 9.71e-09 & 6.54e+07 & 0.050 & 97.48 \\
 &  & 33 & 2.19e-09 & 7.83e-10 & 9.86e-10 & 1.50e+08 & 0.076 & 98.49 \\
 &  & 49 & 8.43e-11 & 3.19e-11 & 3.99e-11 & 4.46e+08 & 0.327 & 99.28 \\
 &  & 57 & 2.23e-11 & 9.10e-12 & 1.13e-11 & 6.67e+08 & 0.444 & 99.46 \\

\bottomrule
\end{tabular}
\end{table}

Table~\ref{tab:scaling_by_rho} reports the errors $e_{\max}$, $e_{\mathrm{mean}}$, and $e_{\mathrm{RMS}}$, together with the condition number $\kappa$, the CPU time, and the sparsity rate of the collocation matrix $\mathcal{A}$. For each fixed pair $(m,\eta)$, the errors decrease as $N_x$ increases, which confirms the convergence of the approximation under grid refinement. Moreover, the results show that the parameter $\eta$ has a direct influence on the balance between computational cost and numerical stability. Larger overlaps improve the stability of the discrete system, whereas smaller overlaps reduce assembly and resolution time. In all reported configurations, the matrix $\mathcal{A}$ remains highly sparse, which is consistent with the local character of the GBMSC construction.

Figure~\ref{fig:scaling_by_eta_emax} shows the log--log convergence curves of $e_{\mathrm{max}}$ as a function of $h={1}/({N_x-1}),$ for the local degrees $m=4$, $m=6$ and $m=8$, and for different values of the overlap parameter $\eta$. The plots show that increasing $\eta$ improves the observed convergence behavior. In particular, the curves associated with larger overlaps are closer to the expected order of convergence $p={m+1}$, whereas the minimal-overlap case $\eta=1$ gives slightly smaller empirical orders. This confirms that the overlap contributes not only to the stability of the algebraic system but also to the effective approximation order.

\begin{figure}
    %%% e_max
        \centering
         \begin{subfigure}[]{0.32\textwidth}
        \centering
        \includegraphics[width=\linewidth]{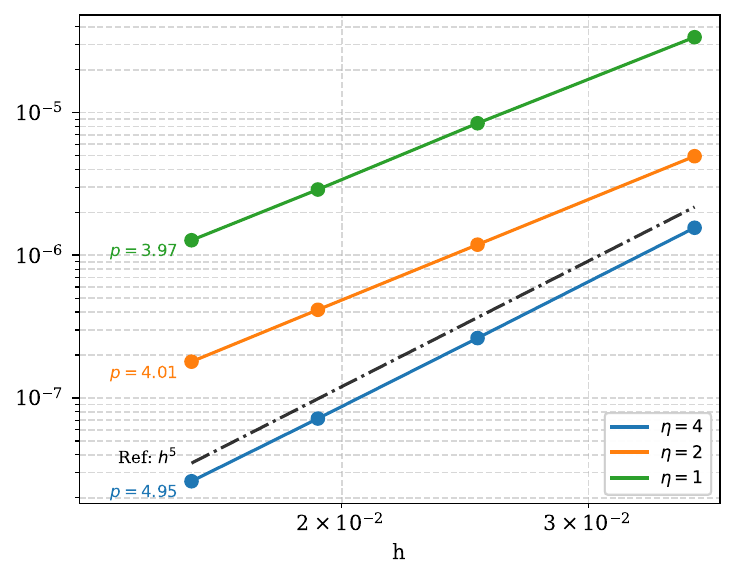}\\[-1mm]
            \caption{$m=4$}
      \end{subfigure}
      \begin{subfigure}[]{0.32\textwidth}
        \centering
    \includegraphics[width=\linewidth]{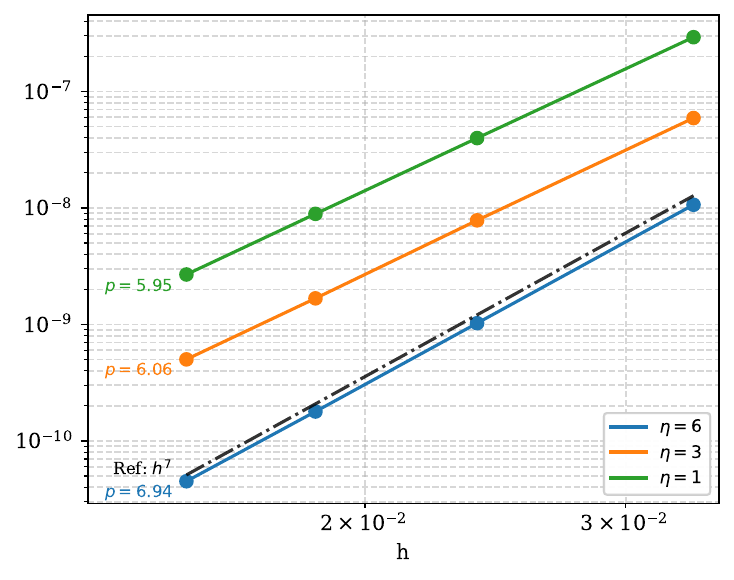}\\[-1mm]
     \caption{$m=6$}
      \end{subfigure}
      \begin{subfigure}[]{0.32\textwidth}
        \centering
    \includegraphics[width=\linewidth]{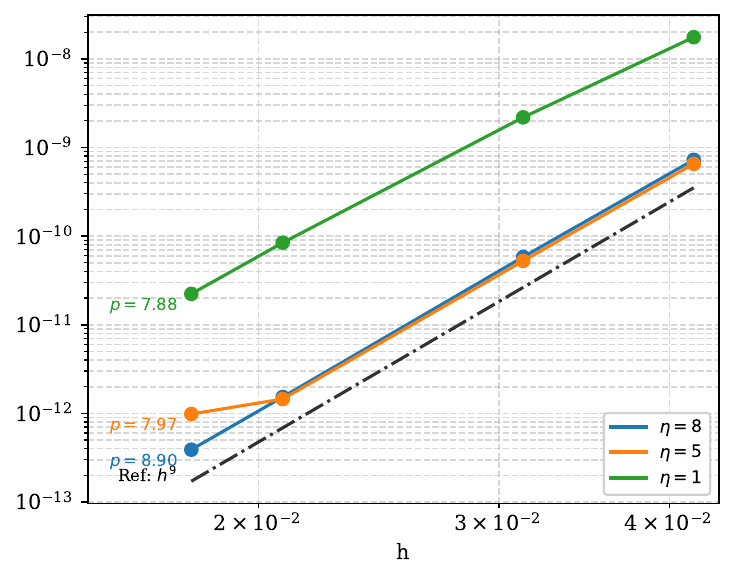}\\
     \caption{$m=8$}
      \end{subfigure}
  \caption{Log--log plots of the error $e_{\text{max}}$ against $h=1/(N_x-1)$ for Problem \ref{Eq.P1_elliptic_Pr} with exact solution $u_2$ across varying polynomial degree $m$ and overlap parameter $\eta$. The values of $p$ represent the empirical convergence order and are obtained by least-squares fitting of $e = C h^p$ on logarithmic data.}
  \label{fig:scaling_by_eta_emax}
\end{figure}

Figure~\ref{fig:wp_comparison} reports a comparison of the tested configurations. More precisely, $e_{\mathrm{mean}}$ is plotted against the CPU time in Figure~\ref{fig:wp_comparison}-(a) and against the condition number $\kappa$ in Figure~\ref{fig:wp_comparison}-(b), both in logarithmic scale. The resolution of the grid $N_x$ is not shown explicitly, but acts as a refinement parameter along each curve. For fixed $(m,\eta)$, increasing $N_x$ decreases the error and increases the computational cost. The same refinement also affects algebraic stability, as $\kappa$ generally increases as the number of collocation points increases.

The plots in Figure~\ref{fig:wp_comparison} indicate that maximal-overlap configurations $(m,m)$ are not always optimal when accuracy is compared with CPU time. Although they provide stable and accurate approximations, their cost is higher. Their position in the error--condition-number diagram also shows that the gain in stability does not always compensate for the additional cost of the larger overlap. The most favorable compromise is obtained by high-degree local approximations with small or moderate overlap. In particular, the configurations $(8,1)$ and $(8,5)$ are located near the most favorable region and generate small errors at a moderate computational cost. Among them, $(8,5)$ gives a more conservative choice when conditioning is also taken into account, whereas $(8,1)$ is preferable when CPU time is the dominant constraint. Thus, for the present Poisson test, high-degree GBMSC discretizations with controlled overlap provide the most efficient choices among the tested configurations.

\begin{figure}
	\centering
     \begin{subfigure}[b]{0.4\textwidth}
    \centering\includegraphics[width=\linewidth]{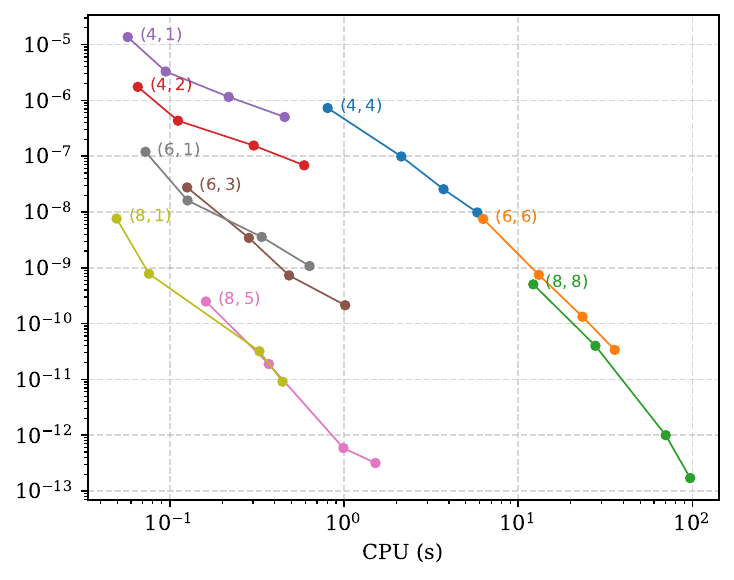}\\[-1mm]
    {\footnotesize (a) $e_{\text{mean}}$ vs CPU}
  \end{subfigure}
  \hspace{0.5cm}
  \begin{subfigure}[b]{0.4\textwidth}
    \centering
    \includegraphics[width=\linewidth]{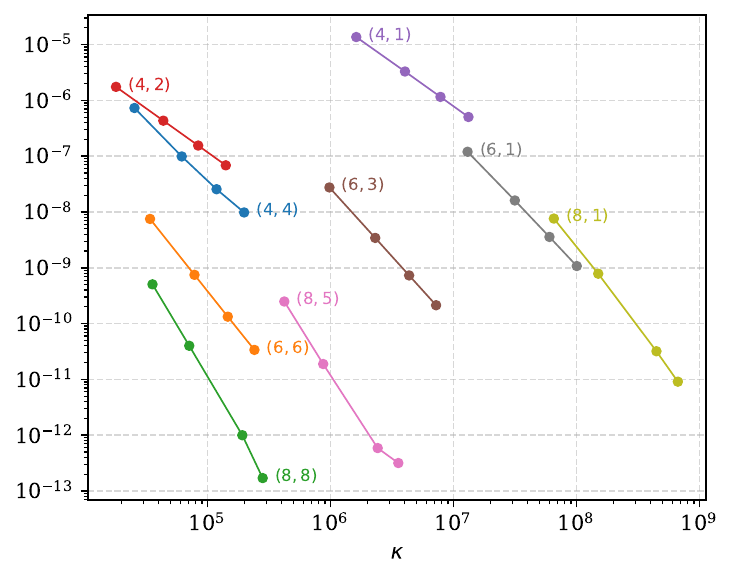}\\[-1mm]
    {\footnotesize (b) $e_{\text{mean}}$ vs $\kappa$}
  \end{subfigure}
  \caption{log-log plots of the error $e_{\text{mean}}$ against CPU, (a), and Condition Number $\kappa$, (b), for Problem \ref{Eq.P1_elliptic_Pr} with exact solution $u_2$. We use the notation $(m,\eta)$; $m$ is the local degree and $\eta$ the overlap parameter.}
  \label{fig:wp_comparison}
\end{figure}

\subsection{Poisson Equation with Mixed Boundary Conditions}
In this second example, we evaluate the accuracy of the numerical scheme for the mixed boundary condition problem of the Poisson equation, as studied in \cite{dell2024multinode, fasshauer2007meshfree}. Here, we consider the following problem on the unit square
$\mathcal D=(0,1)^2$,
\begin{equation}
\label{eq:poisson-mixed-model}
\left\{
\begin{array}{ll}
-\Delta u=f, & \hbox{in } \mathcal D,\\[1mm]
u=g_{\text{Dir}}, & \hbox{on } \Gamma_{\text{Dir}},\\[1mm]
\dfrac{\partial u}{\partial n}=g_{\text{Neu}}, & \hbox{on } \Gamma_{\text{Neu}},
\end{array}
\right.
\end{equation}
where $\partial\mathcal D=\Gamma_{\text{Dir}}\cup\Gamma_{\text{Neu}}$.

For anisotropic discretizations, we use Cartesian grids of size $N_x\times N_y$. The local GBMSC approximation is constructed on subgrids $\mathcal X_k$ of size $n_x\times n_y$. The corresponding local degrees are
\[
m_x=n_x-1\quad\text{and}\quad m_y=n_y-1.
\]
The density of the local subgrids is controlled by the overlap parameters $\eta_x$ and $\eta_y$. This anisotropic construction is useful when the solution has different polynomial structure in the two variables, or when the boundary conditions require a different resolution in the two directions.

In addition to the errors $e_{\max}$, $e_{\mathrm{mean}}$, and $e_{\mathrm{RMS}}$, we measure the accuracy of the imposed Neumann condition. More precisely, on $\Gamma_{\text{Neu}}$ we compute the error
\[
e_{{\text{Neu}}}
=
\left(
\int_{\Gamma_{\text{Neu}}}
\left(
\dfrac{\partial u}{\partial n}-g_{\text{Neu}}
\right)^2\,ds
\right)^{1/2}.
\]
This quantity provides a direct measure of the boundary flux error and complements the pointwise error indicators computed in the domain.

\subsubsection{Problem with Polynomial Solution}
We consider the polynomial manufactured solution $u_3$ \cite{dell2024multinode, fasshauer2007meshfree} of Problem \ref{eq:poisson-mixed-model}:
\[
u_3(x,y)=1-0.9\,x^3,
\]
that yields the source term $f(x,y)= 5.4\,x$ and the functions $g_{\text{Neu}} = 0$ and $g_{\text{Dir}}=u_3\vert_{\Gamma_{\text{Dir}}}$, with
\[
\Gamma_{\text{Dir}}=\{x=0\ \hbox{or}\ x=1\}
\quad\text{and}\quad
\Gamma_{\text{Neu}}=\{y=0\ \hbox{or}\ y=1\}.
\]

Since $u_3$ is independent of $y$ and is a cubic polynomial in $x$, the anisotropic choice $(m_x,m_y)=(3,1)$ is sufficient to reproduce the exact solution in the present setting. Table~\ref{tab:poissonA_rho} confirms the expected polynomial reproduction property. The errors in the domain remain at the round-off level for all tested grids and overlap parameters. The flux error $e_{{\text{Neu}}}$ has the same order of magnitude; hence, the Neumann condition is also reproduced with the expected accuracy. This behavior is consistent with the structure of $u_3$, since the exact normal derivative on $\Gamma_{\text{Neu}}$ is identically zero.

\begin{table}[h]
\centering
\small
\caption{Polynomial reproduction test for the mixed problem \ref{eq:poisson-mixed-model} with exact solution $u_3$. The anisotropic local degree is fixed at $(m_x,m_y)=(3,1)$, while the grid sizes and overlap parameters are varied.}
\label{tab:poissonA_rho}
\begin{tabular}{cccccccccc}
\toprule
$N_x$ & $N_y$ & $\eta_x$ & $\eta_y$
& $e_{\max}$ & $e_{\mathrm{mean}}$ & $e_{\mathrm{RMS}}$
& $e_{{\text{Neu}}}$ & $\kappa$ & Sparsity (\%) \\
\midrule
6 & 5 & 2 & 1 & 1.45e-14 & 6.24e-15 & 7.67e-15 & 6.33e-15 & 3.21e+03 & 87.77 \\
7 & 5 & 3 & 1 & 1.67e-15 & 3.40e-16 & 4.59e-16 & 5.25e-15 & 1.76e+04 & 89.47 \\
10 & 5 & 3 & 1 & 1.55e-15 & 5.30e-16 & 6.14e-16 & 7.54e-15 & 4.08e+04 & 91.84 \\
13 & 5 & 3 & 1 & 2.00e-15 & 6.92e-16 & 8.05e-16 & 6.32e-15 & 7.18e+04 & 93.39 \\
\bottomrule
\end{tabular}
\end{table}

The values of $\kappa$ increase with refinement in the $x$-direction, but remain moderate for this problem. The percentage of sparsity also increases with $N_x$, reflecting the local character of the assembly. Thus, the anisotropic GBMSC construction gives both exact polynomial reproduction up to round-off and a sparse algebraic system.

We compare the results of Table~\ref{tab:poissonA_rho} with those reported in \cite{dell2024multinode}. The comparison is given in Table~\ref{tab:error_comparison}. The GBMSC row corresponds to the representative configuration of Table~\ref{tab:poissonA_rho} for which the method attains round-off accuracy.

\begin{table}[h]
\centering
\caption{Comparison with reference methods for the mixed problem \ref{eq:poisson-mixed-model} with polynomial exact solution $u_3$. The table reports the polynomial degree used in the approximation, collocation points, the discrete errors, and the condition number $\kappa$.}
\label{tab:error_comparison}
\begin{tabular}{lcccccc}
\toprule
Method & Polynomial degree & Collocation points
& $e_{\max}$ & $e_{\mathrm{mean}}$ & $e_{\mathrm{RMS}}$ & $\kappa$ \\
\midrule
GBMSC & $(m_x,m_y)=(3,1)$ & 35
& \textbf{1.67e-15} & \textbf{3.40e-16} & \textbf{4.59e-16} & \textbf{1.76e+04} \\

MSC & $4$ & 353
& 9.57e-14 & 2.70e-14 & 3.53e-14 & 6.41e+05 \\

Kansa M6 (pol) & $4$ & 353
& 3.53e-13 & 7.53e-14 & 1.06e-13 & 1.09e+14 \\

Kansa M6 & -- & 353
& 6.57e-05 & 2.31e-06 & 5.27e-06 & 1.11e+14 \\

Kansa IMQ (pol) & $4$ & 353
& 1.84e-09 & 5.98e-10 & 7.10e-10 & 1.47e+19 \\

Kansa MQ & -- & 353
& 3.19e-06 & 6.33e-07 & 7.91e-07 & 1.55e+19 \\

Kansa IMQ & -- & 353
& 2.51e-06 & 4.60e-07 & 5.96e-07 & 2.28e+19 \\

Kansa MQ (pol) & $4$ & 353
& 9.08e-11 & 2.68e-11 & 3.23e-11 & 5.16e+21 \\
\bottomrule
\end{tabular}
\end{table}

The comparison shows that the GBMSC discretization reaches the smallest errors among the reported methods. The MSC method also provides an accurate approximation, but with a larger condition number. The Kansa's methods show that polynomial enrichment improves the approximation error, but the corresponding matrices remain highly ill-conditioned. This is the main challenge of global RBF collocation.

For the present polynomial test, the relevant point is the simultaneous occurrence of round-off accuracy and moderate conditioning. The GBMSC result not only improves the error level; it also produces an algebraic system whose condition number is significantly smaller than those obtained by the Kansa's methods and smaller than that obtained by the MSC method. This confirms the suitability of the anisotropic GBMSC construction for mixed boundary conditions.

\subsubsection{Problem with Non-Polynomial Solution}
We consider the mixed Poisson problem on $\mathcal D=(0,1)^2$ with exact solution
\[
u_4(x,y)=e^{2x+3y}.
\]
The corresponding right-hand side is
\[
f(x,y)=-13e^{2x+3y},
\]
and the boundary data are defined by
\[
g_{\mathrm{Neu}}=\nabla u_4\cdot \mathbf n_{\Gamma_{\mathrm{Neu}}}
\quad\text{and}\quad
g_{\mathrm{Dir}}=u_4|_{\Gamma_{\mathrm{Dir}}},
\]
where
\[
\Gamma_{\mathrm{Dir}}=\{y=0\ \hbox{or}\ y=1\}
\quad\text{and}\quad
\Gamma_{\mathrm{Neu}}=\{x=0\ \hbox{or}\ x=1\}.
\]
Unlike the previous polynomial test, the function $u_4$ is not exactly reproduced by any local polynomial space of fixed degrees. Hence, this example is used to assess the approximation properties of the GBMSC method for a smooth non-polynomial solution and to compare its accuracy and conditioning with the MSC and Kansa's methods reported in \cite{dell2024multinode}.

The comparison setting is summarized in Table~\ref{tab:params_poisson_mix_nonpoly}. For the GBMSC method, the computations are performed on a Cartesian grid with $580$ collocation points, while the reference methods use $500$ Halton points in the interior and $84$ boundary points. Thus, the numbers of collocation points are comparable, but the node distributions are different. The errors for GBMSC are evaluated on a $101\times101$ Cartesian grid, while the reference results of \cite{dell2024multinode} are reported on a $40\times40$ evaluation grid. Consequently, the comparison must be interpreted as a comparison of error levels and conditioning for the same model problem, rather than as a pointwise reproduction of the same discrete experiment.

\begin{table}[h]
\centering
\caption{Computational setting for the mixed problem \ref{eq:poisson-mixed-model} with exact solution $u_4$. For the reference methods, the parameters are those reported in \cite{dell2024multinode}.}
\label{tab:params_poisson_mix_nonpoly}
\begin{tabular}{lcccc}
\toprule
Method & Node distribution & Collocation points & Polynomial degree & Evaluation grid \\
\midrule
GBMSC & Cartesian grid & $20\times 29=580$ & $m=4,\ldots,11$ & $101\times101$ \\
MSC & Halton + boundary nodes & $500+84$ & $4,\ldots,11$ & $40\times40$ \\
Kansa MQ, IMQ, M6 & Halton + boundary nodes & $500+84$ & -- & $40\times40$ \\
Kansa MQ, IMQ, M6 (pol) & Halton + boundary nodes & $500+84$ & $4,\ldots,11$ & $40\times40$ \\
\bottomrule
\end{tabular}
\end{table}

The comparison with the total-degree $r$ in MSC discretization is carried out in terms of the error indicators and of the conditioning of the resulting linear systems. In Table~\ref{tab:params_poisson_mix_nonpoly}, and in the corresponding numerical comparison, we set $m=r$. This choice is not intended to match the dimensions of the two local polynomial spaces. Rather, it allows us to compare the two constructions under the same nominal polynomial degree and, therefore, to isolate the effect of the local polynomial structure on both accuracy and conditioning.

The results in Table~\ref{tab:GBMSC_nonpoly_mixed} show a clear decrease in error as the degree $m$ increases. The mean error decreases from $2.60\mathrm{e}{-04}$ for $m=4$ to $8.46\mathrm{e}{-11}$ for $m=11$. The same trend is observed for $e_{\max}$ and $e_{\mathrm{RMS}}$. The Neumann-boundary error also decreases, from $7.13\mathrm{e}{-04}$ to $1.88\mathrm{e}{-10}$, which indicates that the discrete normal derivative is consistent with the imposed flux condition. The increase in $\kappa$ with $m$ is expected, as larger local polynomial spaces produce more sensitive local interpolation and differentiation matrices. Nevertheless, the condition number remains below $10^7$ for all reported degrees.

\begin{table}[h]
\centering
\caption{Numerical results of the GBMSC method for the mixed Poisson problem \ref{eq:poisson-mixed-model} with non-polynomial exact solution $u_4$. The table reports the errors in the domain, the Neumann-boundary error, and the condition number $\kappa$ for increasing local polynomial degree $m=m_x=m_y$.}
\label{tab:GBMSC_nonpoly_mixed}
\begin{tabular}{cccccc}
\toprule
$m$ & $e_{\max}$ & $e_{\mathrm{mean}}$ & $e_{\mathrm{RMS}}$
& $e_{{\mathrm{Neu}}}$ & $\kappa$ \\
\midrule
4  & 1.44e-03 & 2.60e-04 & 3.98e-04 & 7.13e-04 & 3.07e+04 \\
5  & 9.09e-04 & 2.06e-04 & 4.66e-04 & 5.89e-05 & 9.88e+04 \\
6  & 1.31e-05 & 2.38e-06 & 3.63e-06 & 5.02e-06 & 6.35e+04 \\
7  & 9.11e-07 & 4.06e-07 & 4.80e-07 & 4.38e-07 & 2.26e+05 \\
8  & 1.24e-07 & 2.26e-08 & 3.42e-08 & 3.90e-08 & 3.34e+05 \\
9  & 1.80e-08 & 5.73e-09 & 6.77e-09 & 3.52e-09 & 8.47e+05 \\
10 & 1.18e-09 & 2.17e-10 & 3.24e-10 & 3.59e-10 & 2.94e+06 \\
11 & 3.59e-10 & 8.46e-11 & 1.07e-10 & 1.88e-10 & 7.57e+06 \\
\bottomrule
\end{tabular}
\end{table}

The comparison with the reference methods is reported in Figure~\ref{fig:nonpoly_comparison}. For small degrees, the GBMSC method has an error level comparable to some Kansa's approximations, while the MSC and polynomial-augmented Kansa MQ are more accurate. From $m=6$, the GBMSC error decreases rapidly. For $m=8$, it is already below the non-augmented Kansa curves and close to the best polynomial-augmented curves. For $m=10$ and $m=11$, the GBMSC method achieves the smallest mean error among the plotted methods.

\begin{figure}
\centering
     \begin{subfigure}[b]{0.4\textwidth}
    \centering
    \includegraphics[width=\linewidth]{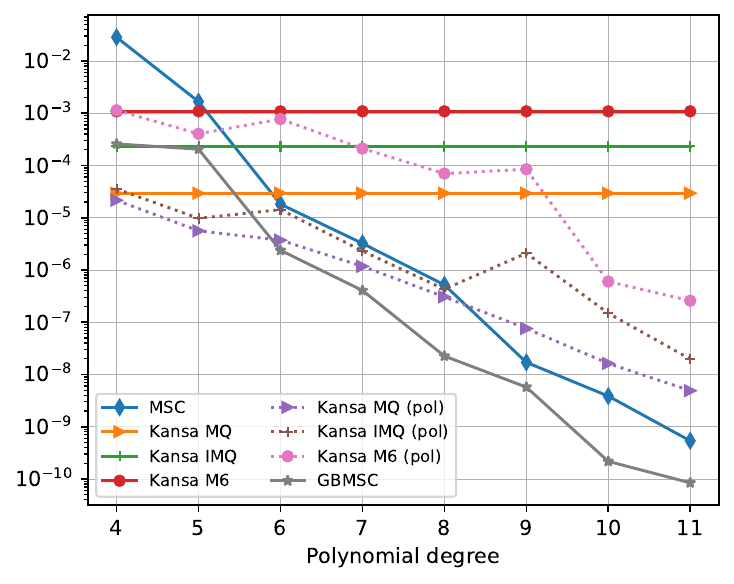}\\[-1mm]
    {\footnotesize (a) $e_{\mathrm{mean}}$ with respect to polynomial degree}
  \end{subfigure}
  \hspace{0.5cm}
  \begin{subfigure}[b]{0.4\textwidth}
    \centering
    \includegraphics[width=\linewidth]{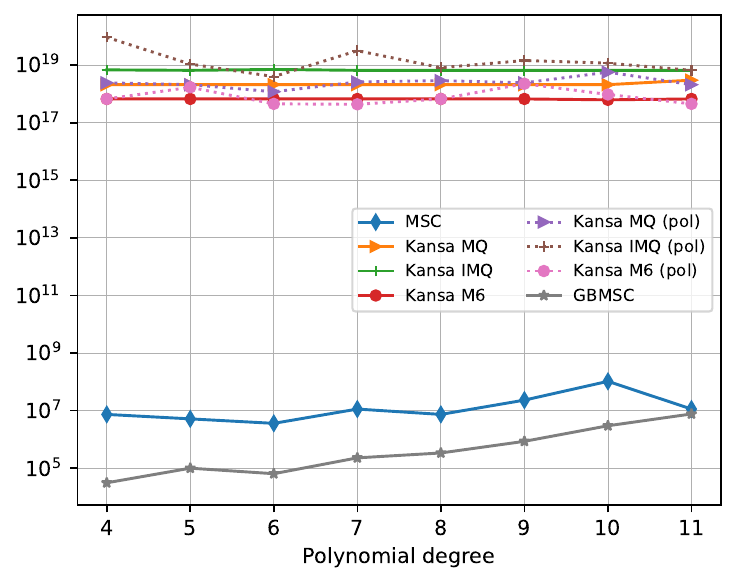}\\[-1mm]
    {\footnotesize (b) $\kappa$ with respect to polynomial degree}
  \end{subfigure}
\caption{Comparison with the reference methods of \cite{dell2024multinode} for the non-polynomial mixed problem \ref{eq:poisson-mixed-model}. Plot (a) reports the mean error, whereas plot (b) reports the condition number as a function of the polynomial degree.}
\label{fig:nonpoly_comparison}
\end{figure}

The conditioning behavior is  different for the considered methods. The Kansa and polynomial-augmented Kansa matrices have condition numbers between $10^{17}$ and $10^{20}$ in the plotted range. This extreme ill-conditioning is consistent with the global nature of the radial basis functions. By contrast, the GBMSC condition numbers remain between $10^4$ and $10^7$, and therefore lie in the same range as the MSC method. In particular, the GBMSC method reaches the smallest errors at high degree while preserving a condition number several orders of magnitude below those of the Kansa's discretizations.

Figure~\ref{fig:nonpoly_GBMSC_convergence}(a) shows the convergence of $e_{{\mathrm{Neu}}}$ under grid refinement. The empirical orders increase with the polynomial degree $m$ and are approximately $4.05$, $4.99$, $5.91$, $6.88$ and $7.63$ for $m=4,\ldots,8$, respectively. These values are consistent with an increase in the local polynomial precision. Figure~\ref{fig:nonpoly_GBMSC_convergence}(b) shows the corresponding growth of the condition number with respect to $\mathrm{DOF}=N_x^2$. The fitted exponents are close to one, which indicates a nearly linear growth of $\kappa$ with the number of degrees of freedom in the tested range.

\begin{figure}
\centering
     \begin{subfigure}[b]{0.4\textwidth}
    \centering
    \includegraphics[width=\linewidth]{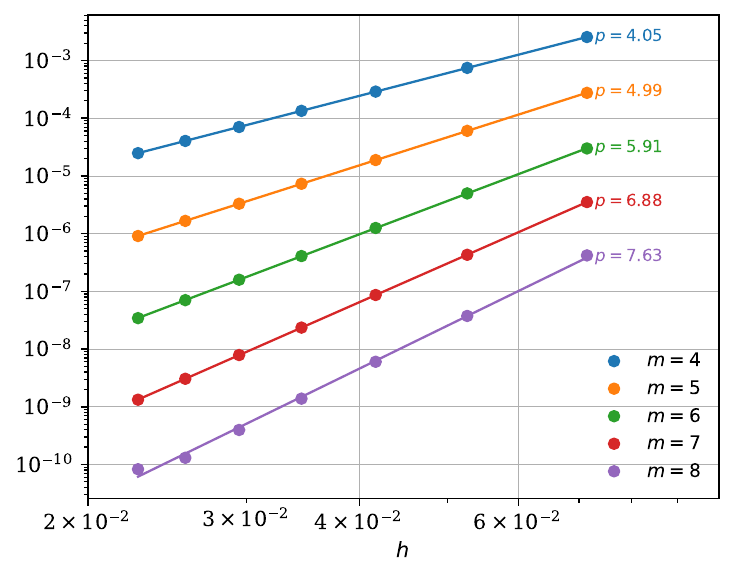}\\[-1mm]
    {\footnotesize (a) $e_{{\mathrm{Neu}}}$ with respect to $h$}
  \end{subfigure}
  \hspace{0.5cm}
  \begin{subfigure}[b]{0.4\textwidth}
    \centering
    \includegraphics[width=\linewidth]{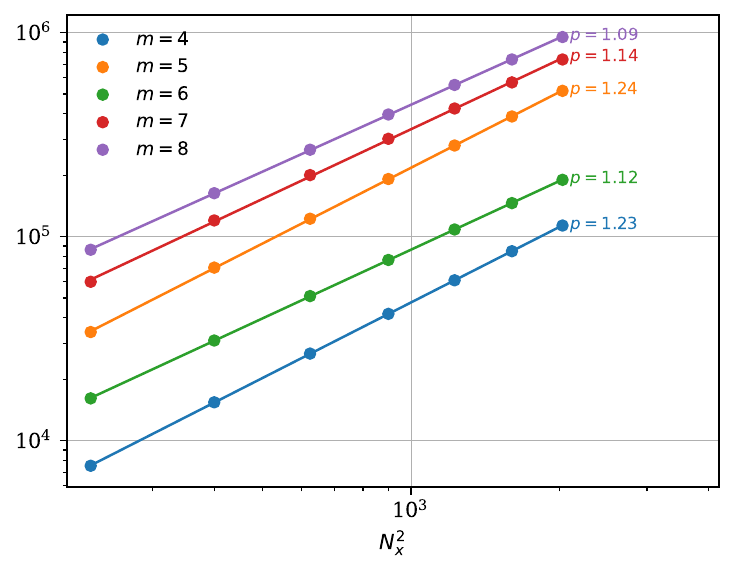}\\[-1mm]
    {\footnotesize (b) $\kappa$ with respect to $\mathrm{DOF}$}
  \end{subfigure}
\caption{Convergence and conditioning of the GBMSC approximation for the non-polynomial mixed problem \ref{eq:poisson-mixed-model}. Plot (a) reports the Neumann-boundary error $e_{{\mathrm{Neu}}}$ as a function of $h$, while plot (b) reports the condition number $\kappa$ as a function of $\mathrm{DOF}=N_x^2$, for different polynomial degrees $m$. The values of $p$ are obtained by least-squares fitting of $e_{{\mathrm{Neu}}}=Ch^p$ and $\kappa=C(\mathrm{DOF})^p$ on logarithmic data.}
\label{fig:nonpoly_GBMSC_convergence}
\end{figure}

It is important to note that the MSC curve remains well conditioned and accurate, especially for intermediate and high degrees. However, the GBMSC curve shows a stronger decay of the mean error at the largest degrees considered here. Thus, for this non-polynomial mixed-boundary test, the GBMSC discretization combines high-order approximation with controlled conditioning. The behavior of $e_{{\mathrm{Neu}}}$ also shows that this accuracy is not restricted to the interior solution values but extends to the normal derivative imposed on the Neumann boundary.

\section*{Conclusion}

In this paper, we introduced the Grid-Based Multinode Shepard Collocation Method (GBMSC) for the numerical solution of two-dimensional elliptic boundary value problems in rectangular domains. The construction is based on overlapping Cartesian subgrids, local Lagrange interpolation, and Shepard-type partition functions. This formulation yields global cardinal basis functions whose derivatives can be assembled from local differentiation matrices and derivatives of the weights.

The numerical results confirm the main properties of the proposed discretization. Polynomial solutions are reproduced with round-off precision throughout the interior of the domain and along the Neumann boundary. For smooth non-polynomial solutions, the errors decrease as the grid is refined and the local polynomial degree is increased. The overlap parameter affects both the accuracy and the conditioning. In fact, larger overlaps improve the accuracy and the conditioning, while smaller overlaps reduce the computational cost but may deteriorate the algebraic stability. The comparison of the GBMSC method with the MSC and Kansa's methods shows that our method presents higher precision and better numerical stability. Future work will extend the GBMSC method to irregular domains.

\section*{Code availability}
A preliminary Python implementation of the proposed method is available in the GitHub repository
\url{https://github.com/e-anouar/gbmsc-pde} archived on Zenodo \cite{el_harrak_2026_gbmsc_pde}. The implementation is publicly available as \texttt{gbmsc-pde}, an installable Python package distributed through PyPI with documentation, examples, and validation tests. The software is distributed
under the Apache License 2.0.

\bibliographystyle{unsrt}
\bibliography{mybibliography}
\end{document}